\documentclass[12pt,leqno]{amsart}
\usepackage{a4,amssymb,amsthm,amscd,amsmath,verbatim,url,enumerate,mathdots,cleveref}
\usepackage[pdftex,usenames,dvipsnames]{color}

\newcommand{\la}{\langle}
\newcommand{\ra}{\rangle}

\newcommand{\nat}{\mathbb{N}}

\newcommand{\zz}{\mathbb Z}
\newcommand{\rr}{\mathbb R}

\newcommand{\mg}[1]{#1^{\times}}
\newcommand{\sq}[1]{#1^{\times 2}}

\DeclareMathOperator{\newsum}{\mathsf{\Sigma}}
\renewcommand{\sum}{\newsum}
\newcommand{\sums}[1]{\sum\!{#1}^2}

\newcommand{\lla}{\la\!\la}
\newcommand{\rra}{\ra\!\ra}

\renewcommand{\inf}{\mathsf{inf}}
\renewcommand{\min}{\mathsf{min}}

\renewcommand{\sup}{\mathsf{sup}}
\newcommand{\half}{\mbox{$\frac{1}2$}}
\renewcommand{\dim}{\mathsf{dim}}

\newcommand\restr[2]{{
  \left.\kern-\nulldelimiterspace 
  #1 
  \vphantom{\big|} 
  \right|_{#2} 
  }}
\newcommand{\car}[1]{\mathsf{char}({#1})}

\renewcommand{\leq}{\leqslant}
\renewcommand{\geq}{\geqslant}

\renewcommand{\bmod}{\,\,\mathsf{mod}\,\,}

\newcommand{\vf}{\varphi}

\newcommand{\W}{\mathsf{W}}
\newcommand{\I}{\mathsf{I}}

\newcommand{\Pfi}{\mathsf{P}}

\newcommand{\tor}{\mathsf{t}}
\newcommand{\st}{\mathsf{st}}
\renewcommand{\dim}{\mathsf{dim}}

\swapnumbers
\numberwithin{equation}{section}
\newtheorem{thm}[equation]{Theorem}
\newtheorem*{thm*}{Theorem}
\newtheorem{prop}[equation]{Proposition}
\newtheorem{cor}[equation]{Corollary}
\newtheorem*{cor*}{Corollary}

\newtheorem*{conj*}{Conjecture}
\newtheorem{lem}[equation]{Lemma}

\newtheorem{qu}[equation]{Question}
\newtheorem*{qu*}{Question}

\theoremstyle{definition}

\newtheorem{ex}[equation]{Example}
\newtheorem{ex*}{Example}

\theoremstyle{plain}

\title{A comparison of finiteness conditions in quadratic form theory}
\date{11.03.2023}
\author{Karim Johannes Becher}
\author{Saurabh Gosavi}
\address{Universiteit Antwerpen, Departement Wiskunde, Middelheim\-laan~1, 2020 Ant\-werpen, Belgium.}
\email{karimjohannes.becher@uantwerpen.be}
\address{Technion-Israel Institute of Technology, Department of Mathematics, Haifa 32000, Israel.}
\email{smgos@campus.technion.ac.il}

\begin{document}


\begin{abstract}
We discuss and relate finiteness conditions for certain field invariants which are studied in quadratic form theory. 
This includes the $u$-invariant, the reduced stability index and the symbol lengths for Galois cohomology groups with  coefficients in $\mu_2=\{+1,-1\}$, as well as a new invariant called the splitting height.


\medskip\noindent
\sc{Classification (MSC 2020):} 11E04, 11E10, 11E81, 19D45
\end{abstract}

\maketitle

\section{The $u$-invariant} 
\label{S:u}

\noindent
A driving incentive in the development of the theory of quadratic forms over fields has been the study of certain   field invariants taking values in $\nat\cup\{\infty\}$.
The most prominent example is the so-called $u$-invariant.
Recent results on the $u$-invariant include its exact determination for certain particular fields, namely for finitely generated extensions of local fields in \cite[Theorem 3.4]{Lee13} and for function fields of surfaces over $\rr$ in \cite[Theorem 0.12]{Ben19}.

The aim of this article is to relate the finiteness of the $u$-invariant to finiteness conditions for other field invariants.
For fields containing a primitive $4$th root of unity, it was shown in \cite[Theorem 5.4]{Kra16} that the finiteness of the $u$-invariant is equivalent to the finiteness of the symbol lengths $\lambda^n$ for all $n\in\nat$.
We will discuss this sequence of invariants in \Cref{S:sl}.
In \Cref{T:main} we will extend this characterisation to all fields of characteristic different from $2$.
In particular, this shall cover the case of (formally) {real fields}, where field orderings are present. 
\smallskip

Let $F$ be a field of characteristic different from $2$. 
We denote by $\mg F$ the multiplicative group of $F$ and by $\sq F$ its subgroup of nonzero squares.
The set of sums of squares in $F$ is denoted by $\sums{F}$. By the Artin-Schreier Theorem  \cite[Theorem VIII.1.10]{Lam05}, $F$ admits a field ordering if and only if $-1\notin\sums{F}$.
We call $F$ \emph{nonreal} if $-1\in \sums F$ and \emph{real} otherwise.

For basic notions and facts from quadratic form theory over fields that is not explained here we refer the reader to the first chapters of \cite{Lam05} and \cite{EKM08}.
The term \emph{form} will always refer to a regular quadratic form.
For a form $\vf$ over $F$ we denote by $[\vf]$ its Witt equivalence class and by $\dim(\vf)$ its dimension (rank).
We denote by $\W F$ the Witt ring of $F$, consisting of classes $[\vf]$ of forms $\vf$ over $F$, and by $\I F$ its fundamental ideal, consisting of the classes of even-dimensional forms.
We denote by $\I_\tor F$ the torsion ideal of $\W F$, consisting of the elements of finite additive order in $\W F$.
By a \emph{torsion form over $F$} we mean a form $\vf$ over $F$ for which $[\vf]\in\I_\tor F$.
Pfister's local-global principle states that the torsion forms over $F$ are precisely the forms having signature zero at all orderings of $F$. 

A form $\vf$ over $F$ is \emph{isotropic} if $\vf$ has a nontrivial zero in $F^n$ where $n=\dim(\vf)$, otherwise it is \emph{anisotropic}.
Following Elman and Lam \cite[Definition 1.1]{EL73}, the \emph{$u$-invariant of $F$} is defined as
\begin{eqnarray*}
u(F) & = & \sup\{\dim(\vf)\mid \vf \mbox{ anisotropic torsion form over }F\}\,\,\in\,\,\nat\cup\{\infty\}.
\end{eqnarray*}
If $F$ is nonreal, then every form over $F$ is a torsion form, and thus $u(F)$ becomes the supremum of the dimensions of all anisotropic forms over~$F$. The results on the $u$-invariant which will be cited from \cite{Kra16} are for nonreal fields, hence where there is no discrepancy in the definition.

\section{Pfister forms and symbol lenghts} 
\label{S:sl}

Let $n\in\nat$. For $a_1,\dots,a_n\in\mg{F}$ we denote by $\lla a_1,\dots,a_n\rra$ the $2^n$-dimensional quadratic form over $F$ given by the tensor product $\la 1,-a_1\ra\otimes \dots\otimes \la 1,-a_n\ra$.
Such quadratic forms are called \emph{$n$-fold Pfister forms}.

We denote by $\Pfi^n(F)$ the set of Witt equivalence classes of $n$-fold Pfister forms over $F$.
It is easy to see that the ideal $\I^n F$ is additively generated by  $\Pfi^n(F)$. 

We denote by $H^n(F)$ the degree-$n$ Galois cohomology group of $F$ with coefficients in $\mu_2 = \{ -1, +1\}$. See \cite{Se97} for the definition and basic properties.
The elements of $H^n(F)$ given by cup products $(a_1) \cup \cdots \cup (a_n)$, with $a_1,\dots,a_n\in\mg{F}$, are called \emph{symbols}.

We rely on the following affirmation of the so-called \emph{Milnor Conjecture}.

\begin{thm}[Orlov-Vishik-Voevodsky]
\label{T:OVV}
There is a natural isomorphism
\[e_n: \I^nF/\I^{n+1}F \rightarrow H^n(F)\]
defined by sending the class of the $n$-fold Pfister form $\lla a_1, \ldots, a_n \rra$ to the symbol $(a_1) \cup \cdots \cup (a_n)$, for any $a_1, \ldots, a_n \in \mg{F}$.
\end{thm}

\begin{proof}
See \cite{Voe03} in combination with \cite[Theorem~4.1]{OVV07} or \cite[Theorem 1.1]{Mor05}.
\end{proof}

By \Cref{T:OVV}, $H^n(F)$ is generated by symbols.  
For $\alpha \in H^n(F)$, we define
\begin{eqnarray*}
    \lambda^n(\alpha) & = & \min\,\{ k \in \nat \mid \alpha = \sum_{i =1}^{k}\beta_i \mbox{ for certain symbols } \beta_1,\dots,\beta_k\mbox{ in }H^n(F)\}.
\end{eqnarray*}

Let $\vf$ be a quadratic form over $F$. 
If $[\vf]\in\I^n F$, then following \cite[Def.~1.3]{Kah05} we define 
\begin{eqnarray*}
\lambda^n(\vf) & = &  \min\,\{k\in\nat\mid [\vf]\equiv \sum_{i=1}^k\rho_i\bmod \I^{n+1} F\mbox{ for certain }\rho_1,\dots,\rho_k\in\Pfi^n(F)\}.
\end{eqnarray*}
If $e_{n}([\vf]) = \alpha \in H^n(F)$, then it follows by \Cref{T:OVV} that $\lambda^n([\vf]) = \lambda^n(\alpha)$.

The following question is a variant of \cite[Question~1.1]{Kah05}.

\begin{qu}\label{Q:generic-effective-bounds}
If $[\vf]\in\I^nF$, can one bound 
$\lambda^n(\vf)$ in terms of $ \dim (\vf)$, and if so, can this be done independently of $F$? 
\end{qu}

A positive answer is only known for $n\leq 3$: see \cite[Prop.~1.1~e)]{Kah05} for $n=2$ and \cite[Cor.~2.1]{Kah05} for $n=3$.
For a given dimension $m\in\nat$, the problem is related to the problem of constructing a generic object (a \emph{versal pair} in the sense of \cite[Def.~5.1]{BF03}) for the functor associating to the field $F$ the set of forms $\vf$ of dimension $m$ with $[\vf]\in\I^n F$. 

For $n\in\nat$, the \emph{degree-$n$ symbol length of $F$} is defined as
\begin{eqnarray*}
\lambda^n(F) & = &  \sup\,\{\lambda^n(\vf) \mid \vf\in\I^n F\}\,.
\end{eqnarray*}
Note that $\lambda^1(F)\leq 1=\lambda^0(F)$.

In view of \Cref{T:OVV}, we have $\lambda^n(F)=\sup\,\{\lambda^n(\alpha)\mid \alpha\in H^n(F)\}$ for any $n\in\nat$.
We refer to \cite{Kah05} for a general discussion of problems  (mostly still open) on the family of invariants $(\lambda^n(F))_{n\in\nat}$ and their relation to the $u$-invariant.
One of those problems is whether $\lambda^2(F)<\infty$ implies that $\lambda^n(F)<\infty$ for all $n\in\nat$.

For $n\in\nat$, we abbreviate $\I^n_\tor F=\I^n F\cap \I_\tor F$.
Let $i,n\in\nat$. We denote by $2^i\times \I^n F$ the ideal of the classes $2^i$-fold multiples of elements of $\I^n F$.
Since $2^i$ corresponds in $\W F$ to the class of the $i$-fold Pfister form $\la 1,1\ra\otimes \dots\otimes\la 1,1\ra$, we have that $2^i\times \I^n F\subseteq \I^{n+i}F$.

In \cite{Bec22}, the following improvement of \cite[Prop.~3.3]{Kah05} was obtained.

\begin{thm} \label{T:KB}
Set $\lambda=\lambda^2(F)$ and assume that $\lambda<\infty$. 
Then  $\I^{2\lambda+2}_\tor F=0$. Furthermore, if $F$ is nonreal, then $\I^{2\lambda+2}F=0$, and 
otherwise $\I^{2\lambda+2} F= 8\times\I^{2\lambda-1}F$. 
\end{thm}
\begin{proof}
See \cite[Theorems 1 and 2]{Bec22}.
\end{proof}

\begin{cor} \label{C:symbol-length-by-divisiblity}
Set $\lambda=\lambda^2(F)$ and assume that $\lambda<\infty$. 
Then $\lambda^{n}(F)=\lambda^{2\lambda+2}(F)$ for any integer $n\geq 2\lambda+2$.
Furthermore, if $\lambda>0$, then $\lambda^{2\lambda+2}(F)\leq \lambda^{2\lambda-1}(F)$. 
\end{cor}
\begin{proof}
    If $F$ is nonreal, then \Cref{T:KB} yields that $\lambda^n(F)=0$ for all integers $n\geq 2\lambda+2$.
    Assume that $F$ is real. Then \Cref{T:KB} yields that $\I^{2\lambda+2}_\tor F=0$ and $\I^n F= 2^{n-2\lambda+1}\times \I^{2\lambda -1}F$ for all integers $n\geq 2\lambda+2$. 
    Hence multiplication by $2$ in $\W F$ defines an isomorphism $\I^{n}F\to \I^{n+1} F$ and a bijection $\Pfi^n(F)\to\Pfi^{n+1}(F)$  for all integers $n\geq 2\lambda +2$.
    For any integer $n\geq 2\lambda + 2$, this induces an isomorphism $\I^{n}F/\I^{n+1} F\to \I^{n+1} F/\I^{n+2}F$ which restricts to a bijection between the respective sets of generators given by the classes of Pfister forms, whence $\lambda^{n+1}(F)=\lambda^n(F)$.
    Furthermore, since $\I^{2\lambda+2} F=8\times \I^{2\lambda-1}F$, multiplication by $8$ in $\W F$ yields a surjective homomorphism $\I^{2\lambda-1}F\to \I^{2\lambda+2}F$ which restricts to a surjection $\Pfi^{2\lambda-1}(F)\to \Pfi^{2\lambda+2}(F)$, and the induced surjection $\I^{2\lambda-1}F/\I^{2\lambda}F\to \I^{2\lambda +2} F/\I^{2\lambda+3} F$ yields that $\lambda^{2\lambda+2}(F)\leq \lambda^{2\lambda-1}(F)$.
    This together shows the statement.
    \end{proof}

\begin{lem}\label{L2}
Let $m\in\nat$. Every class in $\W F/\I^{m+1}F$ is represented by a form $\vf$ with  $\dim(\vf)\leq 1 + \sum_{j = 1}^{m}2^{j}\lambda^j(F)$.
\end{lem}
\begin{proof}
 We may assume that $\lambda^j(F)<\infty$ for $1\leq j\leq m$, as otherwise the claim is trivially satisfied.

We prove the statement by induction on $m$. 
Clearly it holds for $m=0$.
To establish the induction step, suppose now that $m>0$.
Set $\lambda=\lambda^m(F)$.
Consider a form $\vartheta$ over $F$.
By the induction hypothesis for $m-1$, we have that $[\vartheta]=[\psi]\bmod\I^{m}F$ for a form $\psi$ with $\dim(\psi)\leq 1 + \sum_{j = 1}^{m -1}2^j\lambda^j(F)$.
Then $[\vartheta\perp -\psi]\in\I^mF$. 
Since $\lambda=\lambda^m(F)<\infty$, there exist $m$-fold Pfister forms $\rho_1,\dots,\rho_\lambda$ over $F$ such that $[\vf\perp-\psi]\equiv[\rho_1\perp\ldots\perp\rho_\lambda]\bmod\I^{m+1}F$.
Hence, letting $\vf=\psi\perp\rho_1\perp\ldots\perp\rho_{\lambda}$, we obtain that $[\vartheta]\equiv [\vf]\bmod\I^{m+1}F$.
Since we have $\dim(\vf)=\dim(\psi)+2^m\lambda\leq 1+\sum_{j=1}^m2^j\lambda^j(F)$, the statement is proven.
\end{proof}

If $F$ is nonreal, then one sees from \Cref{L2} and \Cref{T:KB} that $u(F)<\infty$ is equivalent with having $\lambda^n(F)<\infty$ for all $n\in\nat$. 
Our main aim is to extend this observation to cover real fields as well. This will be achieved with \Cref{T:main}.

\section{The stability index} 
\label{S:st}

\noindent
Quadratic form theory over real fields combines the aspects of the so-called reduced theory, where sums of squares are treated as if they were squares, with the study of so-called torsion forms. 
The $u$-invariant captures only aspects of the latter. Its study for real fields needs to be combined with the study of the stability index, which takes control of the space of orderings of the field.
In this section we revisit this invariant and its role in connection to the $u$-invariant.
\smallskip

The (\emph{reduced}) \emph{stability index of $F$}, denoted $\st(F)$, is defined as 
\begin{eqnarray*}
\st(F) & = & \inf\,\{n\in\nat\mid \I^{n+1} F=2\times\I^n F + \I^{n+1}_\tor F\}\in\nat\cup\{\infty\}\,.
\end{eqnarray*}
Note that $\st(F)=0$ if and only if $F$ is nonreal or uniquely ordered.

\begin{ex}  
The field of infinitely iterated power series $K= \mathbb{R}(\!(t_1)\!)(\!(t_2)\!)\ldots$ is real with $u(K)=0$ and $\st(K)=\infty$.  
\end{ex}
\begin{thm}[Elman-Lam-Kr\"uskemper]\label{T:ELK} For $n\in\nat$, we have
$\I^{n+1}F(\sqrt{-1})=0$ if and only if $\I^{n+1}_\tor F=0$ and $\st(F)\leq n$.
\end{thm}
\begin{proof}
See \cite{BL11} or \cite[Corollary 35.27]{EKM08}.
\end{proof}

\begin{thm}[Schubert]\label{T:Schubert}
Set $s=\st(F)$. 
If $s=0$, then $u(F(\sqrt{-1}))\leq \frac{3}2u(F) + 1$.
If $0<s<\infty$, then $u(F(\sqrt{-1}))\leq \frac{3}2u(F)+(2^{s+1}-1)2^{s-1}$. 
\end{thm}
\begin{proof}
See \cite[Cor.~1]{Sch07}.
\end{proof}

By \cite[Corollary 2.3]{BG12}, there is an easy bound on the stability index in terms of the $2$-symbol length, which is also known to be optimal.

\begin{thm}[Becher-G\l adki]\label{T:BG}
We have $\st(F)\leq 2\lambda^2(F)-1$.
\end{thm}
\begin{proof}
As an alternative to the proof in \cite[Corollary 2.3]{BG12}, the statement can also be derived directly from \Cref{T:KB}, as explained at the end of \cite[Sect.~1]{Bec22}.
\end{proof}

\section{The splitting height} 
\label{S:sh}

\noindent
It was observed in \cite[Theorem 4.2]{Kra16} that the symbol lengths are controlled by a bound on the minimal degree of finite field extensions that split certain Galois cohomology classes.
Translated to quadratic forms, this motivated to introduce another field invariant, whose study has just started.
\medskip

We say that a form is \emph{split} if it is either hyperbolic or Witt equivalent to a $1$-dimensional quadratic form.
Let $\vf$ be a form over $F$.
Given a field extension $L/F$, we denote by $\vf_L$ the form over $L$ obtained from $\vf$ by scalar extension to $L$. 
Note that every form over $F$ becomes split over an algebraic closure of $F$, and hence already over a finite field extension of $F$.
For a form $\vf$ over $F$, we denote by $\gamma(\vf)$ the smallest degree $[L:F]$ of a field extension $L/F$ such that $\vf_L$ is split, and we call $\gamma(\vf)$ the \emph{splitting height of $\vf$}.

We further set
\begin{eqnarray*}
\gamma(F) & = & \sup\{\gamma(\vf)\mid \vf\mbox{ quadratic form over }F\}\,\,\in\,\,\nat\cup\{\infty\}
\end{eqnarray*}
and we call this the \emph{splitting height of $F$}.

\begin{prop}\label{P:gamma}
For every form $\vf$ over $F$ we have $\gamma(\vf)\leq 2^k$ for $k=\lfloor\half \dim (\vf)\rfloor$.
In particular, if $F$ is nonreal and $u(F)<\infty$, then $\gamma(F)\leq 2^k$ for $k=\lfloor \half u(F)\rfloor$. 
\end{prop}
\begin{proof}
Let $\vf$ be a quadratic form over $F$.
Let $n=\dim(\vf)$ and let  $a_1,\dots,a_n\in\mg{F}$ be such that $\vf=\la a_1,\dots,a_n\ra$.
Set $k=\lfloor \frac{n}{2}\rfloor$ and $K=F(\sqrt{b_1},\dots,\sqrt{b_k})$ where $b_i=-a_{2i-1}a_{2i}$ for $1\leq i\leq k$.
Then $\vf_K$ is split and $\gamma(\vf)\leq [K:F]\leq 2^k$.

Assume now that $F$ is nonreal and $u(F)$ is finite. Set $k=\lfloor \half u(F)\rfloor$.
Consider an arbitrary form $\vf$ over $F$.
Let $\vf'$ be the anisotropic part of $\vf$. Then $\gamma(\vf)=\gamma(\vf')$ and $\dim(\vf')\leq u(F)$, whereby $\gamma(\vf)=\gamma(\vf')\leq 2^k$.
This shows that $\gamma(F)\leq 2^{k}$.
\end{proof}

\begin{cor}\label{C:gamma-bounded-by-u-sqrt-1}
    Assume that $u(F(\sqrt{-1}))<\infty$.
    Let $k=\lfloor \half u(F(\sqrt{-1}))\rfloor$. Then $\gamma(F)\leq 2^{k+1}$.
\end{cor}
\begin{proof}
    This is obvious by \Cref{P:gamma} and the definition of $\gamma$.
\end{proof}

\begin{lem}\label{L3}
    Let $n\in\nat$. For every $\alpha\in H^n(F)$, there exists a separable field extension $K/F$ with 
    $[K:F]\leq \gamma(F)$ such that $\alpha_{K}=0$.
\end{lem}

\begin{proof}
By Theorem \ref{T:OVV}, there exists a form $\psi$ with 
$[\psi] \in \I^{n}F$ and such that $e_{n}([\psi]) = \alpha$. 
There exists a finite field extension $K/F$ with $[K : F] \leq \gamma(F)$ such that $\psi_{K}$ is hyperbolic. Since $\psi_{K}$ is hyperbolic, we have $\alpha_K=e_{n}([\psi_{K}]) = 0$.
Since $\car{F} \neq 2$, taking for $K_0$ the separable closure of $F$ in $K$, the natural map $H^n(K_0)\to H^n(K)$ is injective. Hence we may assume that $K_0=K$ and hence that $K/F$ is separable.
\end{proof}

\begin{thm}[Krashen]\label{P:Krashen}
Assume that, for every $d \in\nat$, there exists $N_d\in\nat$ such that 
$\gamma(F')\leq N_m$ for all field extensions $F'/F$ with $[F':F]\leq m$.
Then $\lambda^{n}(F) < \infty$ for every $n \in \nat$. 
\end{thm}
\begin{proof}
 Let $n\in\nat$.
 The hypothesis implies that,  for any $d,m \in\nat$ with $d<n$ and any field extension $F'/F$ with $[F':F]\leq m$, the effective index as defined in \cite[Sect.~2, Def.~1]{Kra16} of any class in $H^d(F')$ is bounded by $N_m$.

It follows by \cite[Thm.~4.2]{Kra16} that, for any $\alpha\in H^n(F)$, $\lambda^n(\alpha)$ is bounded in terms of $N_n$. Hence $\lambda^n(F)<\infty$ by Theorem \ref{T:OVV}.
\end{proof}

\begin{qu}
    Can one bound $\gamma(F')$ for finite field extensions $F'/F$ in terms of $\gamma(F)$ and $[F':F]$? 
 \end{qu}

\section{Weakly isotropic forms and the main theorem} 

\noindent
Under the assumption that $-1 \in \sq{F}$, it was shown by D.~Krashen \cite[Theorem 5.4]{Kra16} that the finiteness of the $u$-invariant of $F$ is equivalent to the finiteness of the degree-$n$ symbol lengths for all $n\in\nat$.
We will obtain in \Cref{T:main} a generalisation which also covers the case where $F$ is real and in particular does not assume that $-1\in\sq{F}$.
\smallskip

Let $\vf$ be a  form over $F$.
For $n\in\nat$, we denote by $n\times \vf$ the $n$-fold orthogonal sum $\vf\perp\ldots\perp\vf$.
We call $\vf$ \emph{weakly isotropic} if there exists some $n\in\nat$ such that $n\times \vf$ is isotropic, and otherwise we call $\vf$ \emph{strongly anisotropic}. 
Note that any nontrivial torsion form is weakly isotropic.

\begin{lem}\label{L:bigsuboftorsionform}
Let $\rho$ be a torsion form over $F$ and let $\vartheta$ be a subform of $\rho$ with $\dim(\vartheta)>\half\dim(\rho)$. Then $\vartheta$  is weakly isotropic.
\end{lem}
\begin{proof}
Since $\rho$ is a torsion form, there exists a positive integer $k$ such that $k\times \rho$ is hyperbolic. 
Hence $k\times \rho$ contains a totally isotropic subspace of dimension equal to $\half\dim(k\times \vf)$.
Since $k\times \vartheta$ is a subform of $k\times \rho$ and 
$$\dim(k\times \vartheta)=k\cdot\dim(\vartheta)>\half k\cdot\dim(\vf)=\half\dim(k\times \vf)\,,$$ we conclude that $k\times \vartheta$ is isotropic.
\end{proof}

\begin{lem}\label{L1}
Let $n\in\nat$ be such that $\I^n F=2\times\I^{n-1}F$.
Then $\I^{n+1}_\tor F=0$ and every anisotropic form $\psi$ over $F$ with $[\psi]\in \I^{2n-1}F$ is strongly anisotropic.
\end{lem}
\begin{proof}
Since $\I^n F=2\times\I^{n-1}F$, we obtain from \cite[Corollary~XI.4.18~(2) and Theorem~XI.4.5]{Lam05} that  the form $2^n\times \la 1\ra$ over $F$ represents all elements of $\sums{F}$.
By \cite[Proposition 2.4]{Bec06}, this yields that, for any form $\rho$ over $F$, $2^n\times\rho$ is either isotropic or strongly anisotropic.
Since $\I^{2n-1}F=2^n \times\I^{n-1}F$, we conclude that every anisotropic form over $F$ whose class lies in $\I^{2n-1}F$ is strongly anisotropic.
\end{proof}

\begin{prop}\label{P:u-bound-by-lambdas}
Assume that $0<\lambda^2(F)<\infty$.
Then
\[ u(F) \leq 2\left(1 + \sum_{j=1}^{2\lambda^2(F) + 1} 2^{j}\lambda^{j}(F) +2^{2\lambda^2(F) + 2}\big(2^{2\lambda^2(F) + 1} - 1\big)\lambda^{2\lambda^2(F)-1}(F)\right). \]
\end{prop}
\begin{proof}
Set $\lambda=\lambda^2(F)$ and $k= 1 + \sum_{j = 1}^{4\lambda+2}2^{j}\lambda^j(F)$.
By \Cref{T:KB}, we obtain that $\I^{2\lambda+2}F=8\times \I^{2\lambda-1}F=2\times \I^{2\lambda+1}F$.
It follows by \Cref{L1} that every form $\vartheta$ over $F$ with $[\vartheta]\in\I^{4\lambda+3}F$ is strongly anisotropic.

Consider an anisotropic torsion form  $\vf$ over $F$.
By \Cref{L2}, there exists an anisotropic form $\psi$ over $F$ with $\dim(\psi)\leq k$ and such that  $[\vf]\equiv [\psi]\bmod \I^{4\lambda+3} F$.
Let $\vartheta$ be the anisotropic part of $\vf\perp-\psi$.
Then $[\vartheta]\in \I^{4\lambda+3}F$.
Hence $\vartheta$ is strongly anisotropic.
Furthermore, $[\vartheta\perp \psi]=[\vf]\in \I_\tor F$, so $\vartheta\perp \psi$ is a torsion form.
As $\vartheta$ is strongly anisotropic, we obtain by \Cref{L:bigsuboftorsionform} that
$\dim(\vartheta)\leq \half\dim(\vartheta\perp\psi)$.
Since $\dim(\vartheta\perp\psi)\leq \dim(\vartheta)+\dim(\psi)$, 
it follows that $\dim(\vartheta)\leq \dim(\psi)\leq k$.
Since $\vf$ is anisotropic and $[\vf]=[\vartheta\perp\psi]$, we conclude that 
$\dim(\vf)\leq \dim(\vartheta\perp\psi)\leq 2k$.

By \Cref{C:symbol-length-by-divisiblity}, we have $\lambda^n(F)\leq \lambda^{2\lambda -1}(F)$ for every 
integer $n\geq 2\lambda+2$.
Therefore we obtain that $k\leq 1 + \sum_{j=1}^{2\lambda + 1} 2^{j}\lambda^{j}(F) + (\sum_{j=2\lambda+2}^{4\lambda+2}2^j)\lambda^{2\lambda-1}(F)$.
Since $\sum_{j=2\lambda+2}^{4\lambda+2}2^j= 2^{2\lambda+2}(2^{2\lambda+1}-1)$, the statement follows. 
\end{proof}

We now are ready for the main result of this article, which is a generalisation of \cite[Theorem 5.4]{Kra16}.
It extends \cite[Theorem 3]{Sch07}, the equivalence $(i)\Leftrightarrow (ii)$ here.

\begin{thm}\label{T:main}
The following are equivalent:
\begin{enumerate}[$(i)$]
\item $u(F)<\infty$ and $\st(F)<\infty$.
\item $u(F(\sqrt{-1}))<\infty$.
\item $\lambda^n(F)<\infty$ for all $n\in\nat$.
\item $\lambda^n(F)<\infty$ for $2\leq  n\leq 2\lambda^2(F)+1$.
\item For every $d \geq 1$, 
there exists $N_d\in\nat$ such that 
$\gamma(F')\leq N_d$ for all field extensions $F'/F$ with $[F':F]\leq d$.
\end{enumerate}
\end{thm}
\begin{proof}
$(i\Rightarrow ii)$\, This is \Cref{T:Schubert}.

$(ii\Rightarrow v)$\,
Assume that $u(F(\sqrt{-1})) < \infty$. Let $F'/F$ be a finite field extension.
By \cite[Theorem 2.10]{Lee84}, we have that
$u(F^{\prime}(\sqrt{-1}))\leq \frac{1}2(d+1)u(F(\sqrt{-1}))$ where $d=[F':F]$.
We conclude by \Cref{C:gamma-bounded-by-u-sqrt-1} that $\gamma(F^{\prime})\leq 2\gamma(F^{\prime}(\sqrt{-1})) \leq 2^{k+1}$ where $k = \lfloor \frac{1}4(d+1)u(F(\sqrt{-1})) \rfloor$.

$(v\Rightarrow iii)$ This is \Cref{P:Krashen}.

$(iii\Rightarrow iv)$ This implication is trivial.

$(iv\Rightarrow i)$ This follows by \Cref{T:BG} and \Cref{P:u-bound-by-lambdas}.
\end{proof}

\section*{Acknowledgments}
We would like to thank Stephen Scully for helping us with some references. 
We acknowledge support from the \emph{Israel Science Foundation}, grant 353/21.

\bibliographystyle{plain}

\end{document}